\documentclass[11pt]{amsart}
\usepackage[utf8]{inputenc}

\usepackage{xcolor}
\usepackage{hyperref} 
\usepackage{amsfonts, amsthm, amsmath, amssymb,bm}
\usepackage{amscd}
\usepackage{mathtools}

\usepackage{enumerate}

\newcommand{\Zz}{\mathbb{Z}}

\newcommand{\SL}{\mathrm{SL}}

\renewcommand{\epsilon}{\varepsilon}

\newtheorem{thm}{Theorem}[section]
\newtheorem{lem}[thm]{Lemma}

\newtheorem{prop}[thm]{Proposition}
\newtheorem{cor}[thm]{Corollary}
\newtheorem{rmk}{Remark}[section]

\newcommand{\delete}[1]{}

\title[Strong subconvexity]{{\normalfont Strong subconvexity for self-dual $\mathrm{GL} (3)$ $L$-functions}}
\author{Yongxiao Lin}
\address{EPFL/MATH/TAN, Station 8, CH-1015 Lausanne, Switzerland }
\email{yongxiao.lin@epfl.ch}
\author{Ramon Nunes}
\address{Departamento de Matem\'atica, Universidade Federal do Cear\'a, Campus do Pici, Bloco 914, 60440-900 Fortaleza-CE, Brasil \& Max Planck Institute for Mathematics, Vivatsgasse 7, 53111 Bonn, Germany}
\email{ramonmnunes@gmail.com}
\author{Zhi Qi}
\address{School of Mathematical Sciences, Zhejiang University, Hangzhou, 310027, China}
\email{zhi.qi@zju.edu.cn}
\thanks{Y. L. was partially supported by a
	DFG-SNF lead agency program grant (grant
	200020L\_175755) and the SNF (grant 200021\_197045). Z. Q. was supported by NSFC (grant 12071420).}

\usepackage[square,sort,comma,numbers]{natbib}

\begin{document}

	\begin{abstract}
		In this paper, we prove strong subconvexity bounds for self-dual $\mathrm{GL}(3)$ $L$-functions in the $t$-aspect and for $\mathrm{GL}(3)\times\mathrm{GL}(2)$ $L$-functions in the $\mathrm{GL}(2)$-spectral aspect. The bounds are strong in the sense that they are the natural limit of the moment method pioneered by Xiaoqing Li, modulo current knowledge on estimate for the second moment of $\rm GL(3)$ $L$-functions on the critical line.
	\end{abstract}

	\maketitle

	\section{Introduction}

	Let $\phi$ be a {\it self-dual} Hecke--Maass cusp form for $\mathrm{SL}(3,\Zz)$. 
	Let $f_j$ be a Hecke--Maass cusp form for $\mathrm{SL}(2,\Zz)$ for the Laplacian with eigenvalue $1/4+t_j^2$ ($t_j \geq 0$). 
	In this paper, we consider the  subconvexity problem for the $L$-functions $L(s,\phi)$ and $L(s,\phi\times f_j)$.  
	More precisely, for $t$ real, we wish to prove subconvex bounds  of the form
	\begin{equation}\label{eq: subconvex}
		|L(1/2+it,\phi)|^2 \ll_{\phi, \epsilon} (|t|+1)^{\theta  +\epsilon}, \quad  L(1/2,\phi\times f_j)\ll_{\phi, \epsilon} (t_j + 1)^{\theta + \epsilon} , 
	\end{equation}
	for an exponent $ \theta < 3/2$ as small as possible ($\theta = 0$ is the Lindel\"of hypothesis).   
	
	The first result in this problem was obtained by Xiaoqing Li \citep{Li2011bounds} who proved that $\theta = 11/8 $ is admissible in \eqref{eq: subconvex}. Later,  her analysis was refined by McKee,  Haiwei Sun, and  Yangbo Ye \citep{mckee2015improved} and by the second named author \citep{Nunes-GL3} to get the improved exponents $\theta =  4/3 $ and $\theta = 5/4$, respectively.

	Xiaoqing Li obtained a Lindel\"of-on-average upper bound for the first moment of  $L$-functions for $\mathrm{GL}(3) \times \mathrm{GL}(2)$ via the $\mathrm{GL}(2)$ spectral decomposition (Kuznetsov trace formula). In order to deduce individual bounds from the average result, it is crucial that we apply the non-negativity result of Lapid \citep{lapid2003non}: 
	\begin{align}\label{non-negativity}
		L(1/2,\phi\times f_j)\geq 0. 
	\end{align}

	We also mention the work of Munshi  \citep{munshi2015circle}, where he proved the subconvex bound  for $L (1/2+it, \phi)$ as in \eqref{eq: subconvex} with $\theta = 11/8$ without the self-dual assumption on $\phi$. The current record in this general case is  $\theta = 27/20$  due to Aggarwal \citep{Agg2021}. In a preprint \cite{Kumar}, Kumar also announced a subconvexity bound for $L(1/2,\phi\times f_j)$ for general $\phi$ with $\theta = 151/102$.
	
	Since general forms for $\SL(3,\Zz)$ do not necessarily satisfy \eqref{non-negativity}, Munshi, Aggarwal, and Kumar must follow a quite different path---the delta method approach. 
	
	Despite the differences between the two approaches, a common feature of all these works is the central role played by the $\mathrm{GL}(3)$ Voronoi summation formula. 
	
	Finally, Nelson \citep{Nelson} has recently announced an immense breakthrough: He showed a subconvexity bound, albeit with weaker exponents, for all standard $L$-functions on $\mathrm{GL}(N)$ in the case of uniform growth of the spectral parameters (\textit{i.e.}, away from the conductor-dropping case).

	\subsection{Statement of results}

	Our main result is the following bound for the first moment of  $L$-functions for $\mathrm{GL}(3) \times \mathrm{GL}(2)$:

	\begin{thm}\label{GL3}
		For   $\epsilon > 0$ and $T$, $M   $ large  with   $ T^{\epsilon}\leq M   \leq T^{1-\epsilon}$, we have 
		\begin{equation}\label{1eq: main bound}
			\sum_{ |t_j-T|\leq M   } L\left(1/2,\phi\times f_j\right)+ \int_{T-M   }^{T+M   }\left|L\left(1/2+it,\phi\right)\right|^2dt 
			\ll_{\phi,\varepsilon} M    T^{1+\epsilon} + \frac {T^{5/4+\epsilon}} {M   ^{1/4}}.
		\end{equation}  
	\end{thm}
	
	Theorem \ref{GL3} yields averaged Lindel\"of hypothesis as long as $ M    \geq T^{1/5}$, while it was first achieved for $M    \geq T^{3/8}$ in \citep{Li2011bounds}, and later for $M    \geq T^{1/3}$ and $M    \geq T^{5/16}$  in \citep{mckee2015improved} and \citep{Sun-Ye}, respectively. Moreover,   the following bound was proven in  \citep{Nunes-GL3} by the second named author: 
	\begin{equation}\label{average-Nunes}
		\sum_{ |t_j-T|\leq M   } L\left(1/2,\phi\times f_j\right)+ \int_{T-M   }^{T+M   }\left|L\left(1/2+it,\phi\right)\right|^2dt 
		\ll_{\phi,\varepsilon} M    T^{5/4+\epsilon},
	\end{equation}  
	which is not Lindelöf-on-average for any value of $M$ but leads to stronger subconvexity bounds than its predecessors.  
	
	By choosing $M   =T^{1/5}$ in Theorem \ref{GL3} and using the non-negativity in \eqref{non-negativity}, we deduce the following subconvexity bounds. 
	
	\begin{cor}\label{subconv}
		We have the following bounds:
		$$
		L\left(1/2+it,\phi\right) \ll_{\phi, \epsilon} (|t|+1)^{3/5+\epsilon},\quad L\left(1/2,\phi\times f_j\right)\ll_{\phi, \epsilon} (t_j+1)^{6/5+\epsilon}.
		$$ 
	\end{cor}

	\subsection{Remarks on the works of Blomer and Young}
	This work is influenced by the ideas of Blomer  \citep{Blomer2012twisted} and Young  \citep{Young2014weyl}, especially the use of the Archimedean large sieve has its roots in the latter.  
	
	In the most simplified setting, their results can be stated as follows. 
	Let  $\chi_q$ denote the non-trivial quadratic character of prime  conductor $q$. 
	Blomer  proved the bounds 
	\begin{equation*}\label{boundsblomer}
		L\left(1/2+it,\phi\times \chi_q\right) \ll_{\phi,t,\epsilon} q^{5/8+\epsilon},\quad L\left(1/2,\phi\times f_j \times \chi_q\right)\ll_{\phi,f_j,\epsilon} q^{5/4+\epsilon},
	\end{equation*} 
	while Young proved the hybrid bounds
	$$
	L\left(1/2+it,\chi_q\right)\ll_{\epsilon} ((|t|+1) q)^{1/6+\epsilon}, \quad L\left(1/2 , f_j\times \chi_q\right)\ll_{\epsilon} ((t_j+1) q)^{1/3+\epsilon}.
	$$
	
	Note that Blomer's exponent $5/4$ (which agrees with \citep{Nunes-GL3}) is weaker than our $6/5$ in Corollary \ref{subconv}. Even though we do not know how to improve the result of Blomer in the case of prime modulus $q$, the techniques in this paper can be adapted to prove 
	$$
	L\left(1/2+it,\phi\times \chi\right) \ll_{\phi,t,\epsilon} q^{3/5+\epsilon}, 
	$$
	for $\chi $ non-trivial characters of split conductor $q=q_1q_2$ with    $q_1$ and $q_2$ coprime, squarefree and in appropriate ranges. This is a work in progress and will appear in a separate paper.   
	
	Finally, we would like to mention that the idea of using Young's method in this $\mathrm{GL}(3)$ context is also present in the work of Bingrong Huang \citep{huang2016hybrid}. Nevertheless he does not improve on the exponent $2/3$ by McKee--Sun--Ye.

	\subsection{Method of the paper} 
	
	This work  belongs to a long line of papers stemmed from the breakthrough work of Conrey and Iwaniec \citep{CI2000}.

	Our initial steps are standard applications of approximate functional equations and the Kuznetsov formula to the spectral mean of $L$-functions for $\mathrm{GL}(3) \times \mathrm{GL}(2)$. This leads to a diagonal term which is easily dealt with and a more complicated off-diagonal term.
	
	The off-diagonal contribution consists of  $\mathrm{GL}(3)$ Fourier coefficients  weighted by Kloosterman sums. As was the case for all the preceding works, we apply the Voronoi formula to the $\mathrm{GL}(3)$ sum, however we do it in a slightly unusual way which leads to great simplification of the terms we deal with.
	
	If one were to apply  the Voronoi formula of Miller--Schmid \citep{MS2006automorphic}, it would be necessary to open the Kloosterman sum and to create rather cumbersome additional variables. 
	Instead, we observe that one can directly apply a (balanced) Voronoi formula of Miller--Zhou \citep{Miller-Zhou} so that some subsequent calculations as in  \citep{Blomer2012twisted} or \citep{Li2011bounds}   can be completely avoided. 
	

	
	The balanced Voronoi formula of Miller and Zhou was applied successfully in an earlier work of Blomer, Xiaoqing Li, and Miller \citep{Blo-Li-Mil} for a problem involving $\mathrm{GL}(4)$ forms. In our $\mathrm{GL}(3)$ setting, it also render us a simple structural insight on the first moment of $\mathrm{GL} (3) \times \mathrm{GL}(2)$ forms. 
	
	For the next step, we depart from the approach of Xiaoqing Li \citep{Li2011bounds} and follow Young \citep{Young2014weyl} instead.  
	Young  has  prepared the ground for us with the method of stationary phase and Mellin transform. As already observed in Conrey--Iwaniec \citep{CI2000}, the (arithmetic) exponential factor and the (analytic) Hankel transform conspire to make a substantial conductor drop. 
	
	Roughly speaking, after Young's analysis, we arrive at
	\begin{equation}\label{dual-moment1}
		M    \int_{-{T}/{M   }}^{{T}/{M   }}  \bigg( \sum_{r \sim \sqrt{N}/T }\frac{1}{r^{1/2+it}} \bigg) \bigg(\sum_{n\sim \sqrt{N}}\frac{A(1,n)}{n^{1/2-it}}\bigg) \lambda (t) dt,
	\end{equation}
	where $N \leq T^{3+\epsilon}$ is the original length of summation for $L (s, \phi \times f_j)$ or $|L (1/2+it, \phi)|^2$, $A(1, n)$ are the Fourier coefficients of $\phi$, and $\lambda (t)$ is a certain bounded smooth weight.  
	
	Except for the simplification arising from the application of the balanced Voronoi formula, the treatment up to this point is the same as the one in \citep{Nunes-GL3}, at which point, an application of the Archimedean large sieve and Cauchy--Schwarz leads to the result \eqref{average-Nunes}.

	Here lies another novelty of this work: We observe that an application of the functional equation for $\phi$ will be beneficial for the $n$-sum if the length  $\sqrt{N}$ is larger than the square root of the conductor $(T/M   )^3$, as it transforms the sum to a dual sum of shorter length. See \eqref{two-cases} and also Remark \ref{rmk: - case} and \ref{rmk: + case}. 
	
	At last, we conclude by appealing to the aforementioned Archimedean large sieve due to Gallagher and the Cauchy--Schwarz inequality.

	\subsection{Remarks on spectral reciprocity} 
	Based on the looks of the expression in \eqref{dual-moment1} and the approximate functional equation, our approach suggests a relation between 
	$$
	\sum_{ |t_j-T|\leq M   } L\left(1/2,\phi\times f_j\right)+ \int_{T-M   }^{T+M   }\left|L\left(1/2+it,\phi\right)\right|^2dt 
	$$
	and the following   integral of product type $\mathrm{GL}(1) \times \mathrm{GL}(3)$: 
	\begin{align}\label{dual integral}
		M    \int_{-{T}/{M   }}^{{T}/{M   }} \zeta(1/2+it)L(1/2-it,\phi) \lambda (t) dt . 
	\end{align}

	This arises in more explicit terms (with both sides weighted by weights that are related by a certain integral transform)  in the recent work of Chung-Hang Kwan \citep{Kwan}.  His work uses period integral representations of $L$-functions  
	and is in tune with recent studies on \emph{spectral reciprocity formulae}. 
	
	It seems that an alternative version of Kwan's formula can still be derived by the   ``Kuznetsov--Voronoi" approach with more elaborate analysis, but this of course is not needed for the subconvexity problem. Theoretically, it is suspected that if Kwan's weight function  is properly chosen and his integral transform is carefully analyzed---the analysis will probably resemble the stationary-phase analysis of Young---then our results might be deduced from his formula.  
	
	Finally we remark that a third proof of such a spectral reciprocity formula is possible, using the Kuznetsov formula and analytic continuation of Dirichlet series; see a forthcoming work by Humphries and Khan.

	\subsection{Strength of our results}
	
	In view of the above discussion, the term $ T^{5/4+\epsilon} / M   ^{1/4} $ in  Theorem \ref{GL3} is tied to  the following (trivial) large-sieve estimate for the second moment of $L (1/2+it, \phi)$:  \begin{align}\label{2nd moment}
		\int_{-U}^{U} \left|L(1/2+it, \phi)\right|^2dt\ll_{\phi} U^{3/2+\epsilon}.
	\end{align} 
	Any improvement over \eqref{2nd moment} directly leads to better subconvexity exponents for $L(1/2+it,\phi)$ and $L(1/2,\phi\times f_j)$. However, to our knowledge, showing such an improvement is a wide open problem. Therefore  we have reached the natural limit of this moment approach, and our subconvexity bounds are strong in some sense.

	Finally, we remark that if  
	the cusp form $\phi$ were replaced by a maximal or minimal parabolic Eisenstein series for $\mathrm{SL}(3,\Zz)$, the $L$-function in \eqref{dual integral} can be factored and a sharper bound will follow from a second moment of $\mathrm{GL}(2)$ $L$-functions for which we know optimal results. In these cases one may get the Lindel\"of-on-average bound $M    T^{1+\epsilon}$ for the integral. In fact, the minimal parabolic situation corresponds to a particular case of the result of Young \citep{Young2014weyl}.

	\section*{Acknowledgments}
	We would like to thank Philippe Michel for his encouragement at the various stages of this project. We are grateful to Peter Humphries for his comments and to the referee for his/her detailed comments and helpful suggestions. We also thank Roman Holowinsky and Ritabrata Munshi for their interests in this work.

	\section{Preliminaries}
	
	We first briefly review the central tools in this paper---the Kuznetsov trace formula for $\mathrm{GL}(2)$, 
	the functional equation and	the  Voronoi summation formula for $\mathrm{GL}(3)$. We refer the reader to \citep{CI2000,GoldfeldGLn,MS2006automorphic,Miller-Zhou} for more details. 
	
	\subsection{Kuznetsov trace formula}
	
	Let $\{f_j\}_{j \geq 1}$ be an orthonormal basis of even Hecke--Maass forms for $\mathrm{SL} (2, \Zz)$. For each $f_j (z)  $ with Laplacian eigenvalue $ 1 / 4 + t_j^2$ ($t_j \geq 0$), it has Fourier expansion of the form
	\begin{align*}
		f_j(z)= 2 \sum_{n\neq 0}\rho_j(n)\sqrt{|n| y} K_{it_j}(2\pi|n|y)e(nx) .
	\end{align*}
	Let $\lambda_j (n)$ ($n \geq 1$) be its Hecke eigenvalues. 
	It is known that $\rho_f (\pm n) = \rho_f (1) \lambda_f (n) / \sqrt{n}$.

	The Kuznetsov trace formula for even Maass form for $\mathrm{SL} (2, \Zz)$ reads as follows (see \cite[(3.17)]{CI2000}). 
	\begin{lem}\label{Kuznetsov}
		Let  $h (t)$ be an even test function satisfying the conditions:
		\begin{enumerate} 
			\item[{\rm (i)}] $h (t)$ is holomorphic in  $|\operatorname{Im}(t)|< {1}/{2}+\epsilon$,
			\item[{\rm (ii)}] $h(t)\ll (|t|+1)^{-2-\epsilon}$ in the above strip. 
		\end{enumerate}
		Then for $n_1, n_2 \geq  1$ 	we have the following  identity: 
		\begin{equation}
			\begin{split}
				\sum_{j \geq 1} h(t_j)\omega_j   \lambda_j(n_1) &  \lambda_j(n_2)+\frac{1}{4\pi}\int_{-\infty}^{\infty}h(t)\omega(t)\tau_{it}(n_1)\tau_{it}(n_2)dt\\
				&= \delta_{n_1,n_2}H +  \sum_{\pm}\sum_{c= 1}^{\infty}\frac{S(n_1,\pm n_2;c)}{c}H^{\pm}\left(\frac{4\pi\sqrt{n_1n_2}}{c}\right),
			\end{split}
		\end{equation}
		where  $\delta_{n_1,n_2}$ is the Kronecker $\delta$-symbol, 
		\begin{align}
			\tau_{s} (n) = \tau_{-s} (n) = \sum_{ab=n} (a/b)^s, 
		\end{align}
		\begin{equation}\label{omegaj}
			\omega_j=\frac{4\pi |\rho_j(1)|^2}{\cosh(\pi t_j)}, \qquad \omega (t) = \frac {4\pi} {|\zeta(1+2it)|^2},
		\end{equation}
		and
		\begin{align}\label{DBB}
			\begin{aligned}
				H&=\frac{1}{\pi}\displaystyle\int_{-\infty}^{\infty}h(t)\tanh(\pi t)tdt,\\
				H^+(x)&= {i}   \int_{-\infty}^{\infty}h(t) \frac{J_{2it}(x)}{\cosh(\pi t)}t dt,\\
				H^-(x)&=\frac{2}{\pi}\int_{-\infty}^{\infty}h(t) K_{2it}(x)\sinh(t)tdt .
			\end{aligned}
		\end{align} 
		
	\end{lem}
	
	The spectral weights $\omega_j$ and $\omega (t)$ play a very minor role in our problem in view of (see \cite[Theorem 2]{Iwaniec90} and \cite[Theorem 5.16]{Titchmarsh} respectively)
	\begin{align}\label{2eq: lower bound of omega}
		\omega_j\gg t_j^{-\epsilon}, \qquad \omega(t)\gg t^{-\epsilon} . 
	\end{align}

	\subsection{\texorpdfstring{$\mathrm{GL}(3)$ Maass forms and their $L$-function}{GL(3) Maass forms and their {\it L}-function}}\label{lfunctions}
	
	We refer the reader to  Goldfeld's book \citep{GoldfeldGLn} for the theory of $\mathrm{GL}(3)$ Maass forms.  
	
	Let $\phi$ be a {\it self-dual} Hecke--Maass form for $\mathrm{SL}(3,\Zz)$ of Langlands parameter $(\mu, 0, -\mu)$ and Hecke eigenvalues $A(n_1,n_2)$($=\overline{A(n_1, n_2)}= A(n_2, n_1)$). 
	For later use, we record here the  Rankin--Selberg estimate:
	\begin{equation}\label{RS}
		\sum_{n\leq X}|A(1,n)|^2\ll_{\phi} X.
	\end{equation} 
	
	Define the $L$-function attached to $\phi$ by
	\begin{align}\label{eq: L-series}
		L(s,\phi)=\sum_{n=1}^{\infty}\frac{A(1,n)}{n^s},
	\end{align} 
	for $\mathrm{Re} (s) > 1$, and by analytic continuation for all $s$ in the complex plane. Next, we define  the 
	gamma factor
	\begin{align}\label{eq: gamma}
		\gamma (s, \phi) = \pi^{-3s/2} \Gamma \Big(\frac {s-\mu}2\Big) \Gamma \Big(\frac {s}2\Big) \Gamma \Big(\frac {s+\mu}2\Big). 
	\end{align}
	The functional equation for $ L(s,\phi) $ reads
	\begin{align}\label{functional-eq}
	\gamma (s, \phi) L(s, \phi) =\gamma (1-s, \phi)  L (1-s, \phi). 
	\end{align}
	
	A consequence of the functional equation is the following summation formula, commonly known among specialists.
	
	\begin{lem}\label{summation from FE} For $t, N > 1$ and fixed $w \in C_c^{\infty} (0,\infty)$ we have the identity
		\begin{equation} \label{afe}
			\sum_{n= 1}^{\infty} \frac{A(1,n)}{n^{1/2-it}} w\left( n/N\right)=\sum_{n= 1}^{\infty}\frac{A(1,n)}{n^{1/2+it}}W (nN; t) ,
		\end{equation}
		with 
		\begin{equation}\label{eq: W(y;t) = }
			W (y; t)= \frac 1 {2\pi i} \int_{(0)}  \frac{\gamma(1/2+it-s,\phi)}{\gamma(1/2-it+s,\phi)}\widetilde{w}(s) y^{s} ds, 
		\end{equation} 
		where $\widetilde{w} $ is the Mellin transform of $w $. Moreover, we have
		\begin{align}\label{bound for Wt(y)}
			W (y; t) \ll_{A, \mu,  w} \left( 1 + \frac {y} {t^{3}} \right)^{-A} . 
		\end{align}
	\end{lem}
	
	\begin{proof}
	We shall only give a sketch of proof following the arguments as in the proofs of Theorem 5.3 and Proposition 5.4 in \citep{IK-analytic}.

	First, in view of \eqref{eq: L-series}, we may use   the inverse Mellin transform to rewrite	the left hand side of \eqref{afe}  as
		\begin{align*}
	\frac{1}{2\pi i}\int_{(3)}\widetilde{w}(s)N^sL(1/2-it+s,\phi)ds.
	\end{align*}
	Next, we shift the integral contour  to $\mathrm{Re} (s)= - 3$ (note that $\widetilde{w}(s) $ decays rapidly in $\mathrm{Im}(s)$) and then apply the functional equation \eqref{functional-eq}, giving
		\begin{align*}
	\frac{1}{2\pi i}\int_{(-3)}\widetilde{w}(s)N^s\frac{\gamma(1/2+it-s,\phi)}{\gamma(1/2-it+s,\phi)}L(1/2+it-s,\phi)ds.
	\end{align*}
	This gives us the right hand side of \eqref{afe} by inserting \eqref{eq: L-series} and shifting the integral contour to $\mathrm{Re}(s) = 0$.  
	
	Finally, the gamma quotient in \eqref{eq: W(y;t) = } is of unity norm  whenever $\mathrm{Re}(s)=0$ due to \eqref{eq: gamma}, so it is clear that $W (y; t)$ is always bounded, while $W (y; t) \ll (y/t^3)^{-A} $ follows from  Stirling's formula and by shifting the integral contour in \eqref{eq: W(y;t) = } to the far left $\mathrm{Re}(s) = - A$. Therefore  the estimate in \eqref{bound for Wt(y)} follows.

	\end{proof}
	
	The estimate in \eqref{bound for Wt(y)} indicates that the sum on the right of \eqref{afe}  can be effectively truncated at $ t^{3+\epsilon}/ N$.

	\subsection{\texorpdfstring{Reversed  Voronoi summation formula for $\mathrm{GL}(3)$}{Reversed  Voronoi summation formula for GL(3)}}

	The Voronoi formula for   $\mathrm{GL}(3)$  in the next lemma may be considered as the reverse of that of Miller and Schmid \citep{MS2006automorphic}. It is also a special case of the balanced Voronoi formula of Miller and Fan Zhou \citep[Theorem 1.1]{Miller-Zhou}.
	For simplicity, we again assume that $\phi$ is self-dual. 
	\begin{lem}\label{voronoi}
		
		For $w  \in C_c^{\infty} (0, \infty)$ define its Hankel transform $W $ by
		\begin{align}
			W (\pm y) =  \frac{1}{2\pi i}\int_{(-3)}G^{\pm}(s)\widetilde{w}(s) y^{s-1}ds  ,
		\end{align}
		where  $\widetilde{w} $ is the Mellin transform of $w $, and
		\begin{equation}
			G^{\pm}(s)=\frac{\gamma(1-s, {\phi})}{\gamma(s,{\phi})}\pm i^3 \frac{\gamma(2-s, {\phi})}{\gamma(1+s, {\phi})} . 
		\end{equation}
		Let $a, \overline{a}, c, m$ be integers with   $a \overline{a} \equiv 1 (\mathrm{mod}\, c)$ and $c, m > 0$.	Then    we have 
		\begin{equation}
			\begin{split}
				\sum_{n_2\mid cm}\sum_{n_1=1}^{\infty}n_2 A(n_1,n_2)& S\left(n_1,\overline{a}m;cm/n_2\right) w(n_1n_2^2)\\
				& =\sum_{\pm}\sum_{n=1}^{\infty}\frac{A(m,n)}{c}e\left(\pm \frac{{a}n}{c}\right){W} \left(\pm \frac{n}{c^3m}\right).  
			\end{split}
		\end{equation} 
	\end{lem}
	
	According to \citep[\S \S 3.3, 14]{Qi-Bessel}, there is a Bessel kernel $J_{\phi} (x)$  attached to $\phi$  so that the Hankel transform may indeed be realized as an integral transform
	\begin{align}
		W (y) = \int_{0}^{\infty} w (x) J_{\phi} (- x y) d x,
	\end{align}
	and  the following asymptotic expansion holds:
	\begin{align}\label{4eq: asymptotic, Bessel, R} 
		J_{\phi}  (\pm x ) & =     \frac { { e  \left(   \pm  3 x^{1/3}  \right)   }}  {x^{1/3} }   \sum_{k= 0}^{K-1} \frac {B^{\pm}_{k } }  {x^{  k/3  }}   +  O_{K, \mu} \bigg( \frac 1 {x^{ (K + 1) / 3}} \bigg), 
	\end{align}
	for $x > 1$, where $B^{\pm}_{k }$ are some constants depending on the spectral parameters of $\phi$. Observe that the Bessel kernel	$J_{\phi} (x)$ has the same type of asymptotic as the double integral in \citep[Lemma 6.5]{Young2014weyl} arising from triple Fourier transform.  
	

	\subsection{Approximate functional equations} 
	The $L$-function $L(s, \phi)$ for the $\mathrm{GL}(3)$ Maass form $\phi$ has been introduced in \S \ref{lfunctions}, and it is known that for $\mathrm{Re}(s)>1$,
	\begin{align}
		L(s+it,\phi) L(s-it, \phi)=  \sum_{n_1=1}^{\infty} \sum_{n_2=1}^{\infty}\frac{A(n_1,n_2)\tau_{it}(n_1)}{(n_1 n_2^2)^{s}} . 
	\end{align} The other main actor is the $L$-function of the Rankin--Selberg convolution $   \phi \times f_j$ given, for $\mathrm{Re}(s)>1$, by
	\begin{equation}\label{RSL-function}
		L(s,\phi\times f_j)=\sum_{n_1=1}^{\infty} \sum_{n_2=1}^{\infty}\frac{A(n_1,n_2)\lambda_j (n_1)}{(n_1 n_2^2)^s}.  
	\end{equation}
	According to \citep[Theorem 5.3]{IK-analytic}, since $f_j$ is assumed to be even, we have the following approximate functional equations: 
	
	\begin{align}\label{AFE-Maass}
		L(1/2,\phi\times f_j)=2\sum_{n_1=1}^{\infty} \sum_{n_2=1}^{\infty}\frac{A(n_1,n_2)\lambda_j (n_1)}{(n_1 n_2^2)^{1/2}} V (n_1 n_2^2; t_j ),
	\end{align}
	and
	\begin{align}\label{AFE-Eisenstein}
		\left|L(1/2+it,\phi)\right|^2=2 \sum_{n_1=1}^{\infty} \sum_{n_2=1}^{\infty} \frac{A(n_1,n_2)\tau_{it}(n_1)}{(n_1 n_2^2)^{1/2}} V (n_1 n_2^2; t) , 
	\end{align} 
	where
	\begin{align}
		V  (y; t) = \frac 1 {2\pi i} \int_{(3)}  G  (u, t) y^{-u} \frac {du} {u},   
	\end{align} 
	and 
	\begin{align}
		G (u, t) = \frac {\gamma (1/2+it +u, \phi) \gamma (1/2-it +u, \phi) } {\gamma (1/2+it , \phi) \gamma (1/2-it , \phi) } \cdot \exp(u^2). 
	\end{align}
	
	The following lemma is essentially from \cite[Proposition 5.4]{IK-analytic} and \cite[Lemma 1]{Blomer2012twisted}. 
	\begin{lem}\label{short-int-over-u}  
		Let  $ U  > 1 $.  We have 
		\begin{align}\label{bound for Vt(y)}
			V  (y; t ) \ll_{ A }  
			\bigg(  1 + \frac {y} { (|t|+1)^3  } \bigg)^{-A}, 
		\end{align}
		and
		\begin{align}
			\label{approx of Vt(y)} 
			V (y; t) & = \frac 1 {2   \pi i } \int_{ \epsilon - i U}^{\epsilon + i U}  G  (u, t)    y^{ - u}   \frac {d u} {u} + O_{\epsilon } \bigg( \frac {(|t|+1)^{  \epsilon} } {y^{ \epsilon} e^{U^2/2  } } \bigg) .
		\end{align}  
		
	\end{lem}

	\subsection{The large sieve} 
	
	The following   is  Gallagher's large sieve inequality \citep[Theorem 2]{gallagher1970large} in the case $q=1$.
	
	\begin{lem}\label{lem: large sieve}
		Let $a_n$ be a sequence of complex numbers.	For $T > 1$ we have
		\begin{align}
			\int_{-T}^T \bigg|\sum_{n} a_n n^{it} \bigg|^2 d t \ll \sum_{n} (T + n) |a_n|^2, 
		\end{align}
		provided that the sum of $|a_n|$ is bounded. 
	\end{lem}
	
	\section{Initial steps}

	The initial steps---applications of Kuznetsov and Voronoi---are now considered standard in the literature. However, we depart from previous work in that the Voronoi formula is used in the reversed direction. This crucially simplifies subsequent computations.

	Consider the smoothed spectral mean of $L$-values 
	\begin{align*}
		\mathcal{M}=\sum_{j } k(t_j)\omega_jL\left({1}/{2},\phi\times f_j\right)+ 
		\frac{1}{4\pi}\int_{-\infty}^{\infty}k(t)\omega(t)\left|L\left({1}/{2}+it,\phi\right)\right|^2dt, 
	\end{align*}
	where the spectral weights $\omega_j$ and $\omega (t)$ are given in \eqref{omegaj}, and the test function $k(t)$ is defined by
	\begin{equation*}
		k(t)=e^{- (t  - T)^2 / M   ^2} + e^{-(t + T)^2 / M   ^2}.
	\end{equation*}
	In view of \eqref{non-negativity} and \eqref{2eq: lower bound of omega},  the bound \eqref{1eq: main bound} in Theorem \ref{GL3} follows if we can prove 
	\begin{equation}\label{whatwillprove}
		\mathcal{M}\ll M    T^{1+\epsilon} + \frac {T^{5/4+\epsilon}}{M   ^{1/4}}. 
	\end{equation}
	Next, we use the approximate functional equations \eqref{AFE-Maass} and \eqref{AFE-Eisenstein} to write 
	\begin{align*}
		\mathcal{M} = 2	\sum_{n_1=1}^{\infty} \sum_{n_2=1}^{\infty} \frac {A(n_1, n_2)} {(n_1 n_2^2)^{1/2}} \Bigg( & \sum_{j} k(t_j)\omega_j  \lambda_j(n_1) V  (n_1n_2^2; t_j )\\
		&	+\frac{1}{4\pi}\int_{-\infty}^{\infty}k(t)\omega(t) \tau_{it}(n_1) V  (n_1n_2^2; t) dt \Bigg).
	\end{align*}
	Moreover, in view of \eqref{bound for Vt(y)}  in Lemma \ref{short-int-over-u}, we can truncate the $(n_1, n_2)$-sum at $n_1 n_2^2 \leq T^{3+\epsilon}$,  at the cost of a negligible error.  
	
	\subsection{Applying the Kuznetsov formula} 
	Applying the Kuznetsov formula as in Lemma \ref{Kuznetsov} with $n_2=1$, we obtain certain diagonal   and off-diagonal sums. 
	
	The diagonal sum reads
	\begin{align*}
		\mathcal{D} = \frac {2}{\pi} \sum_{  n_2 \leq  T^{3/2 +\epsilon} } \frac {A(1, n_2)} {  n_2 }  \displaystyle\int_{-\infty}^{\infty}k(t) V  ( n_2^2; t  ) \tanh(\pi t)tdt .
	\end{align*}
	It follows from \eqref{RS}, \eqref{bound for Vt(y)}, and  Cauchy--Schwarz,  that
	\begin{align*}
		\mathcal{D} \ll M    T^{1+\epsilon}, 
	\end{align*}
	which is the first term on the right of \eqref{whatwillprove}.

	For the off-diagonal sum, some preparation is required before we proceed to use the Voronoi formula. First, following Blomer \cite[\S 5]{Blomer2012twisted}, we express $V (n_1 n_2^2, t)$ in the Bessel integrals by \eqref{approx of Vt(y)} in Lemma \ref{short-int-over-u}, choosing $U = \log T$ so that the error term is negligible.  Second, we dyadically decompose the variable $n =n_1 n_2^2$. Third, we introduce a new variable $r = c n_2$. Our task is then reduced to consider sums of the   form 
	\begin{equation}\label{eq: Spm}
		\mathcal{S}_{\pm}(N)=\hskip -1pt   \sum_{r=1}^{\infty}\frac{1}{r}\sum_{n_2\mid r}  \sum_{n_1=1}^{\infty}n_2 A(n_1,n_2)  S( n_1, \pm 1;r/n_2)w_{\pm}\bigg(\frac{n_1 n_2^2}{N};\frac{\sqrt{ N}}{r}\bigg),
	\end{equation} 
	for $1/2 < N \leq T^{3+\epsilon}$, where 
	\begin{align*}
		w_{\pm}(x;D)= w (x) x^{-u} H^{\pm}(4\pi D\sqrt{x}) , 
	\end{align*}
	with  $w (x) \in C_c^{\infty} [1,2]$ satisfying $w^{(j)} (x)\ll_j 1$, and  
	\begin{align*}
		H^{+} (x ) & = {i}   \displaystyle\int_{-\infty}^{\infty}k(t) G (u, t)  \frac{ J_{2it}(x)}{\cosh(\pi t)}t dt,\\ 	H^-(x )&=\frac{2}{\pi}\displaystyle\int_{-\infty}^{\infty}k(t) G (u, t) K_{2it}(x)\sinh(t) tdt,
	\end{align*}
	for  $u\in [\epsilon-i\log T,\epsilon+i \log T ]$. 
	It turns out that the analysis and the final estimate will be uniform in $u$, so we have suppressed $u$ from our notation.

	To get the bound $ T^{5/4+\epsilon}/M   ^{1/4}$ as in \eqref{whatwillprove} for the off-diagonal contribution, it suffices to prove the following bound for $\mathcal{S}_{\pm} (N)$. 
	
	\begin{prop}\label{mustprove}
		For any  $N \leq T^{3+\epsilon}$	we have uniformly  
		\begin{equation*}
			\frac {\mathcal{S}_{\pm}(N ) } {N^{1/2} }\ll_{\phi,\epsilon} \frac{T^{5/4+\epsilon}}{M   ^{1/4}} . 
		\end{equation*} 
	\end{prop}
	
	Finally, we remark that the $r$-sum can be truncated as (the Bessel integral in) $w_{\pm}(x;D)$  is negligibly small unless 
	\begin{align}\label{D<T}
		D \gg T . 
	\end{align}
	See \cite[\S 8.1]{Qi-GL(3)}, \cite[\S 7]{Qi-Liu-Moments}, and also \cite[\S \S 4, 5]{Li2011bounds}, \cite[\S 7]{Young2014weyl} in slightly different settings. This condition can be refined as in \eqref{condition +} and \eqref{condition -} (see Remark \ref{rem: conditions}).
	
	\subsection{Applying the Voronoi formula}

	We apply the Voronoi summation formula in Lemma \ref{voronoi}, with $a = \pm 1$, $c=r$, and $m=1$, for the $(n_1, n_2)$-sum in \eqref{eq: Spm}, obtaining
	\begin{align}\label{eq: S(N)}
		\mathcal{S}_{\pm}(N)&= N \sum_{r > 0} \frac{1}{r^2}\sum_{n\neq 0}A(1,|n|)e\left(\pm  \frac{n}{r}\right)  W_{\pm} \bigg(\frac{N n}{r^3};\frac{\sqrt{N }}{r}\bigg),  
	\end{align} 
	where $ {W}_{\pm} (y;D)$ is the Hankel transform of $w_{\pm}(x,D)$. This expression is much simpler than that of Xiaoqing Li \citep{Li2011bounds} after applying the (first) Voronoi.

	\section{Results on the Hankel transform} 
	
	It turns out that the exponential factor and the Hankel transform in \eqref{eq: S(N)} combine perfectly by substantially reducing the oscillation of the latter.   Accordingly, define 
	\begin{equation}\label{W-tilde}
		\widetilde{W}_{\pm} (y; D) = e (\pm y/D^2) {W}_{\pm} (y; D).
	\end{equation}
	This kind of combination can be traced back to the work of Conrey and Iwaniec \citep{CI2000}. 
	
	The lemma below provides a nice expression of $\widetilde{W}_{\pm} (y; D) $ designed for the large sieve in the next stage. It is   proven by Young in \cite[Lemma 8.2]{Young2014weyl} using the method of stationary phase and Mellin transform. In the $\mathrm{GL}(3)$ setting, the reader is referred to \cite[\S 4]{huang2016hybrid} and \cite[\S \S 10--13]{Qi-GL(3)} for more details (the latter also contains the complex analogue,  which is needed if we work over general number fields).

	\begin{lem}\label{integralrep}
		Let $T^{\epsilon} \leq M    \leq T^{1-\epsilon}$.	Suppose  that $|y| > T^{\epsilon}$. For $  y  \asymp X D^2$  we may write 
		\begin{align}\label{eq: W = Phi}
			\widetilde{W}_{\pm} (y; D) = \frac {M    T^{1+\epsilon}} {\sqrt{|y|}} \Phi^{\pm} (y/D^2) + O_A (T^{-A}), 
		\end{align} 
		such that $\Phi^{+} (x ) = 0$ and $ \Phi^{-} (x ) = 0 $ unless
		\begin{align}\label{condition +}
			T < D/  M   ^{1-\epsilon} , \qquad |X| \asymp D, 
		\end{align}
		and
		\begin{align}\label{condition -}
			|X| < D / M   ^{3-\epsilon}, \qquad  T \asymp D, 
		\end{align}
		respectively, in which cases 
		\begin{align}\label{eq: Phi pm}
			\Phi^{\pm} (  x) = \frac 1 {T} \int_{|t| \asymp U^{\pm}} \lambda_{X, T}^{\pm} (t) |x|^{it} d t, 
		\end{align}
		with $\lambda_{X, T}^{\pm} (t) \ll 1$ {\rm(}$\lambda_{X, T}^{\pm} (t)$ depends on $X$, $T$ but the implied constant does not{\rm)} and 
		\begin{align}
			U^+ = T^2/|X|, \qquad U^- = |X|^{1/3} T^{2/3},
		\end{align}
		provided that $|X| < T^{2-\epsilon}$. 
	\end{lem}
	
	\begin{rmk}
	    \label{rem: conditions}
	The conditions $T < D/M^{1-\epsilon}$ and $T \asymp D$ in \eqref{condition +} and \eqref{condition -} arise in Lemma 7.1 and 7.2 in \cite{Young2014weyl} for the analysis of Bessel integrals. The conditions $|X| \asymp D$ and $|X| < D / M   ^{3-\epsilon}$ in \eqref{condition +} and \eqref{condition -} come from the discussions after (8.13) in \cite{Young2014weyl} or Corollary 11.10 in \cite{Qi-GL(3)}. 
	\end{rmk}
	
	\begin{rmk}
	    For the   case  $|X| \geq T^{2-\epsilon}$,   $U^{+} = T^2/|X|$ needs to be replaced by $U^{+} = T^{\epsilon}$. See \cite[Corollary 11.10, Lemma 13.2]{Qi-GL(3)} for more details. However, as it will be seen, this extreme case does not occur over $\mathbb{Q}$.
	\end{rmk}

	The condition $|y| > T^{\epsilon}$ is needed for applying the asymptotic formula \eqref{4eq: asymptotic, Bessel, R}.  In our setting, $N \leq T^{3+\epsilon}$, and $$y = Nn/r^3, \qquad D = \sqrt{N}/ r. $$ Therefore $ |y| =  |n| D^3/ \sqrt{N} > T^{3/2-\epsilon}$ in view of $D \gg T$ as in \eqref{D<T}.
	
	For the $+$ and $-$ case,  the conditions in \eqref{condition +} or \eqref{condition -} amount to 
	\begin{align}\label{eq: range of r, n, +}
		r < \frac { \sqrt{N} } {M   ^{1-\epsilon} T}, \qquad |n| \asymp \sqrt{N},
	\end{align} or
	\begin{align}\label{eq: range of r, n, -}
		\ \   r \asymp \frac { \sqrt{N} } {T}, \qquad \ \ \ \ |n| < \frac {\sqrt{N}}{M   ^{3-\epsilon}}, 
	\end{align}
	respectively. In particular,  $|X| < T^{2-\epsilon}$ is   obviously satisfied as $|X| \asymp |n / r| \ll \sqrt{N} <  T^{3/2+\epsilon}$. Moreover,  under the conditions  in \eqref{condition +} or \eqref{condition -},  we have 
	\begin{align}\label{U+, U- < U}
		T^{\epsilon}	< U^{\pm} < \frac {T} {M   ^{1-\epsilon}} ,
	\end{align}
	for $ T^{\epsilon}/D^2 < |X| < T^{2-\epsilon} $.

	\section{End game: the large sieve}
	By \eqref{eq: S(N)}, \eqref{W-tilde}, and \eqref{eq: W = Phi}, we may write 
	\begin{align}\label{eq: S(N), 2}
		\mathcal{S}_{\pm}(N) =  M    \sqrt{N}  T^{1+\epsilon}  \sum_{r>0}\frac{1}{\sqrt{r}}\sum_{n\neq 0} \frac{A(1,|n|)}{\sqrt{|n|}} \Phi^{\pm} \Big(\frac {n} {r} \Big) + O_{A} (T^{-A}).  
	\end{align} 
	By applying \eqref{eq: Phi pm} and a dyadic decomposition of the variables $r$ and $n$, we deduce that 
	$\mathcal{S}_{\pm}(N) / \sqrt{N}$ is bounded by the supremum of the norm of
	\begin{equation}\label{almostthere}
		M      T^{\epsilon}   \int_{|t| \asymp U} \lambda^{\pm} (t) \sum_{r \sim R^{\pm}}\frac{1}{\sqrt{r} }\sum_{n \sim N^{\pm}}\frac{A(1,n)}{\sqrt{n} } w \Big(\frac n {N^{\pm}}\Big)\Big(\frac{n}{r}\Big)^{it} dt  ,
	\end{equation} for parameters $R^{\pm}$, $N^{\pm}$, and $U$ in the ranges (see \eqref{eq: range of r, n, +}, \eqref{eq: range of r, n, -}, and \eqref{U+, U- < U})
	\begin{equation}\label{URN} 
		\begin{split}
			& 1/2 < R^+ <    {\sqrt{N}}  / {M   ^{1-\epsilon} T}, \qquad   \   N^{+} \asymp \sqrt{N}, \\
			& R^- \asymp  {\sqrt{N}} / {T}, \qquad   1/2 < N^{-} <  {\sqrt{N}} / {M   ^{3-\epsilon}}  , \\
			& \qquad \qquad \ \    T^{\epsilon} <  U <    {T}  / {M   ^{1-\epsilon} }. 
		\end{split}	   
	\end{equation}
	Here again $w (x) \in C_c^{\infty} [1,2]$ and $\lambda^{\pm} (t)$ is the $\lambda_{X, T}^{\pm} (t)$ in Lemma \ref{integralrep} with subscripts suppressed for simplicity.  
	For the $n$-sum in \eqref{almostthere}, we can further apply Lemma \ref{summation from FE} (the functional equation for $\phi$) to get
	\begin{align}\label{two-cases}
		\sum_{n \sim N^{\pm}}\frac{A(1,n)}{n^{1/2-it}} w ( n /N^{\pm}) = \sum_{n \leq   {U^{3+\epsilon}}/{N^{\pm}}}\frac{A(1,n)}{n^{1/2+it}} W (N^{\pm} n, t)+O(T^{-A}),
	\end{align}
	for $|t| \asymp U$, so that the length of summation can be shortened if $N^{\pm} > U^{3/2}$. 
	
	Finally, by Cauchy--Schwarz and Lemma \ref{lem: large sieve}, along with the Rankin--Selberg bound for $A(1, n)$ in \eqref{RS}, we infer that \eqref{almostthere} may be bounded by
	\begin{align}\label{final bound}
		M       T^{\epsilon} \sqrt{R^{\pm}+U}  \sqrt{    \min   \left\{
			N^{\pm} , U^3/N^{\pm} \right\} + U  }   ,
	\end{align} 
	where   $\min   \{
	N^{\pm}, U^3/N^{\pm}  \}$ is due to  the observation in \eqref{two-cases}. 
	
	A subtle issue is that once \eqref{two-cases} were applied, the $W (N^{\pm} n; t)$ on the right is dependent on $t$. To address this, we use the expression \eqref{eq: W(y;t) = } to separate the variables $n$ and $t$---this can be done effectively with the $s$-integral truncated at $|\mathrm{Im}(s) | = T^{\epsilon}$ and  ${\gamma(1/2+it-s,\phi)}/{\gamma(1/2-it+s,\phi)}$ absorbed into $\lambda^{\pm} (t)$ (note that  $\widetilde{w}(s)$ is of rapid decay while the gamma quotient is of unity norm for $\mathrm{Re}(s)=0$).  
	
	Trivially, we use $\min \{
	N^{+}, \allowbreak U^3/N^{+} \} \leq U^{3/2}$ and $\min \{
	N^{-}, \allowbreak U^3/N^{-} \} \leq N^{-}$. In view of \eqref{URN} and the inequality $N \leq T^{3+\epsilon}$, \eqref{final bound} is further bounded by
	\begin{align*}
		M       T^{\epsilon} \bigg(\frac {\sqrt{N}} {M    T} +   \frac{T}{M   }\bigg)^{1/2}\bigg(\frac{T}{M   }\bigg)^{3/4} \ll \frac {T^{5/4+\epsilon}} {M   ^{1/4}}, 
	\end{align*}
	in the $+$ case, and 
	\begin{align*}
		M       T^{\epsilon} \bigg(\frac {\sqrt{N}} {T} +\frac{T}{M   }\bigg)^{1/2}\bigg(\frac{\sqrt{N}}{M   ^3}\bigg)^{1/2} \ll \frac {T^{1+\epsilon}} {M   ^{1/2}} + \frac {T^{5/4+\epsilon}} {M    },  
	\end{align*}
	in the $-$ case, 
	as desired by Proposition \ref{mustprove}.

	\begin{rmk}\label{rmk: - case}
		For the $-$ case, we did not need to apply the functional equation {\rm(}Lemma \ref{summation from FE}{\rm)}---in the same way that Xiaoqing Li \citep{Li2011bounds} did not use a second Voronoi. Note that $ N^{-} < U^{3/2} $ in the typical case when $N^{-} \asymp \sqrt{N}/ M   ^{3 }$ and $U \asymp T / M    $. 
	\end{rmk}
	
	\begin{rmk}\label{rmk: + case}
		For the $+$ case, the worst case does not occur when $N \asymp T^{3}$ {\rm(}the   transition range{\rm)} and $N^{+}\asymp T^{3/2}$. In fact, the possibility of applying functional equation with $|t|\asymp T/M$ after \eqref{almostthere} makes the range $N\asymp (T/M)^{3 }$ the one   with the largest contribution. 
	\end{rmk}


\end{document}